%% file: CVhybride.tex
\documentclass[a4paper, reqno,11pt]{amsart}
\usepackage[utf8]{inputenc}
\usepackage[T1]{fontenc}
\usepackage[english]{babel}
\usepackage{todonotes}
\usepackage{babelbib}
\usepackage{amsmath}
\usepackage{amsfonts}
\usepackage{amssymb}
\usepackage{amsthm}
\usepackage{color}
\usepackage{amsthm}
\usepackage{graphicx}
\usepackage{hyperref}
\usepackage{textcomp}
\usepackage{array}
\usepackage{pgf}
\usepackage{tikz}
\usetikzlibrary{matrix,arrows}
\usepackage{mathrsfs}
\usepackage{bbm,yhmath,stmaryrd}
\usepackage{geometry}
\newtheorem{defn}{Definition}

\newtheorem{theo}{Theorem}
\newtheorem*{theoA}{Theorem A}
\newtheorem*{theoB}{Theorem B}

\newtheorem{prop}{Proposition}
\newtheorem{lem}{Lemma}

\newtheorem{ex}{Example}
\newtheorem{qu}{Question}

\newtheorem{rem}{Remark}
\numberwithin{theo}{section}
\numberwithin{defn}{section}
\numberwithin{conj}{section}
\numberwithin{prop}{section}
\numberwithin{lem}{section}
\numberwithin{defprop}{section}
\numberwithin{cor}{section}
\numberwithin{qu}{section}
\numberwithin{rem}{section}
\numberwithin{ex}{section}
\input macros2rue.tex
\title{Hybrid convergence of Kähler-Einstein measures}
\author{Léonard Pille-Schneider}
\begin{document}
\date{}
\nocite{*}
\maketitle
\begin{abstract}
We compute the hybrid limit (in the sense of Boucksom-Jonsson) of the family of Kähler-Einstein volume forms on a degeneration of canonically polarized manifolds.
\\The limit measure is a weighted sum of Dirac masses at divisorial valuations, determined by the natural algebro-geometric limit of the family.
\\We also make some remarks on the non-archimedean Monge-Ampère operator and hybrid continuity of Kähler-Einstein potentials in this context.
\end{abstract}
\tableofcontents
\begin{section}*{Introduction}
Let $X \xrightarrow{\pi} \de^*$ be a holomorphic family of $n$-dimensional compact Kähler-Einstein manifolds, where $\de^* =\{ \lvert t \rvert <1 \}$ is the punctured unit disk in $\C$. We moreover assume that the family $X$ has a meromorphic singularity at $t=0$, in the sense that there exists a normal \emph{model} $\X \rightarrow \de$ of $X$. By definition, this is a normal complex analytic space $\X$ together with a flat, proper holomorphic map $\pi : \X \fl \de$, such that $X = \pi^{-1}(\de^*)$. By Hironaka's theorem, up to a resolution of singularities, we can always assume that the model is simple normal crossing: the irreducible components of the central fiber are smooth and intersect each other transversally.
\\For simplicity, we will also assume that the family has semi-stable reduction, which means that we can find a simple normal crossing model with reduced central fiber - this condition can always be achieved after a finite base change $t \mapsto t^m$ by the semi-stable reduction theorem \cite{KKMSD}.
\\The Kähler-Einstein assumption means that each fiber $X_t$ carries a Kähler-Einstein metric, unique in its Kähler class, satisfying the complex Monge-Ampère equation :
$$\Ric (\omega_t) = \lambda \omega_t,$$
with $\Ric(\omega_t) = -\frac{i}{2 \pi} \ddbar \log \omega_t^n$, and with Einstein constant $\lambda \in \{ 0, \pm 1\}$ (up to rescaling the metric). If $\lambda = \pm 1$, the Kähler class $[\omega_t] \in H^2(X_t, \Z)$ is determined by $X$, given by $\lambda c_1(X_t)$, and if $\lambda =0$, we will furthermore assume that our family is relatively polarized by a line bundle $L$, so that in any case $\omega_t \in c_1(L_t)$, with $L = -\lambda K_{X/\de^*}$ when $\lambda = \pm 1$, ensuring the uniqueness of $\omega_t$ (up to isometry when $\lambda =1$).
\\It is a classical problem in Kähler geometry to understand the limiting behaviour of these metrics as $t$ approaches $0$, and in particular to relate the metric limit to the degeneration of the complex structure on $X_t$, and to the birational geometry of models of $X$. 
\\This problem appears to be particularly involved in the Calabi-Yau case, where $\lambda =0$. In this case, the diameters of the metric spaces $(X_t, g_{\omega_t})$ may blow up: they stay bounded if and only if the family can be filled to a normal family $\X /\de$ such that the central fiber $\X_0$ is a mildy singular (canonical) Calabi-Yau variety, see for instance \cite{To}. The volume of the metric being cohomological and thus constant, in general some collapsing occurs and it is necessary to renormalize the metrics if one wishes to find a Gromov-Hausdorff limit of the fibers - which is no longer an algebro-geometric object.
\\A similar, and related question is to understand the limiting behaviour of the Kähler-Einstein \emph{measures} $\mu_t =\omega_t^n$ on $X_t$ (which have constant mass $L_t^n$). This question is of particular interest in the Calabi-Yau case, where $\mu_t$ turns out to be, up to a scaling factor, of algebraic nature. Indeed, if we fix a holomorphic trivialisation $\Omega$ of the relative canonical bundle, the Ricci-flatness equation can be rewritten as:
$$\omega_t^n = C_t i^{n^2} \Omega_t \wedge \bar{\Omega}_t,$$
where we denote by $\Omega_t$ the fiberwise restriction of $\Omega$, and $C_t$ a positive constant making the integrals of the two sides match.
In this set-up, understanding the limiting behaviour of the measures $\mu_t$ seems to be a natural first step for understanding the more involved behaviour of the metrics.
\\In this paper, we will be interested in the convergence of the Kähler-Einstein measures in the \emph{hybrid} set-up, which we will now describe.
\\The family $X/ \de^*$ being projective, one can canonically associate to $X$ a locally compact topological space $X^{\text{hyb}} \xrightarrow{\pi} \de$ mapping to the disk, called the \emph{hybrid space}, whose restriction to the punctured disk is the complex analytic space $X$, and whose central fiber $X^{\text{hyb}}_0$ is the non-archimedean Berkovich space $X^{\text{an}}$ associated to the base change of $X$ to the non-archimedean field $K = \C((t))$. This hybrid space, whose construction goes back in spirit to Berkovich, was defined in \cite{BJ} and provides a natural way to study degenerations of measures on a degeneration of complex manifolds, as it provides a canonical fill-in of the family $X/ \de^*$, whose central fiber has a reasonable topology. Convergence of measures coming from holomorphic dynamical systems has for instance been studied in the same set-up in \cite{Fav}.
\\In the Calabi-Yau case, this hybrid space is used by Boucksom-Jonsson \cite{BJ} to prove that the Kähler-Einstein measures (as mentioned above, those are nothing but a rescaling of $i^{n^2} \Omega_t \wedge \bar{\Omega}_t$ here) converge weakly on $X^{\text{hyb}}$ to a measure supported on $X^{\text{an}}$, which is a Lebesgue-type measure on a canonical simplicial complex $\Sk(X) \subset X^{\text{an}}$ called the \emph{essential skeleton}, and thus provides in a natural sense a \emph{tropical} limit of the $\mu_t$'s -  see section 1.2 for a precise statement.
\\The purpose of this note is to provide a similar result for degenerations of Kähler-Einstein manifolds with negative Ricci curvature. In this case, the limiting behaviour of the Kähler-Einstein metrics is fairly well-understood - this goes back to Tian \cite{T} under certain technical assumptions, with recent extension to the general case by Song \cite{S}. In fact, the metrics are non-collapsing as $t \rightarrow 0$ on a large region of the $X_t$'s, and the metric limit is actually closely related to the "algebro-geometric limit" of the $X_t$'s: indeed, in this case, the family $X/ \de^*$ admits a distinguished model $\X_c / \de$, whose central fiber matches the metric limit in a natural sense, and on which the Kähler-Einstein measures admit a weak limit in the complex analytic topology.
\\This in turn will imply that the tropical limiting behaviour of the Kähler-Einstein is of very different nature than in the Calabi-Yau case:
\begin{theoA}
Let $X \fl \de^*$ be a meromorphic degeneration of $n$-dimensional canonically polarized manifolds, and let $\omega_t \in -c_1(X_t)$ be the unique Kähler-Einstein metric with negative curvature on $X_t$. We furthermore assume that the family has semi-stable reduction.
\\The Kähler-Einstein measures $\mu_t = \omega_t^n$ converge weakly on $X^{\emph{hyb}}$ to a finite sum $\mu_0$ of Dirac masses supported at the divisorial valuations $v_{D_i}$ corresponding to the irreducible components $(D_i)_{i \in I}$ of the canonical model $\X_c$ of $X$, in the sense of the Minimal Model Program.
\\More precisely, we have :
$$ \mu_0 = \sum_{i \in I} (( K_{\X_c})^n \cdot D_i) \delta_{v_i}.$$
Moreover, $\mu_0 = \MA( \phi_{K_{\X_c}})$ is the non-archimedean Monge-Ampère measure associated to the model metric induced by the line bundle $K_{\X_c}$ on $\X_c$ (see section 1.1 for definitions).
\end{theoA}
For instance, if $X$ is a degeneration of curves of genus $g \ge 2$, which we assume to have semi-stable reduction, then the limiting measure is naturally supported on the vertex set of the dual graph $\Gamma$ of the special fiber $\X_0 = \sum_{i \in I} C_i$ of the minimal model $\X/ \de$. By adjunction, $\mu_0$ puts a mass $\mu_0(v_i) = 2 g(C_i) -2 + \val(v_i)$ on the vertex $v_i$, of valence $\val(v_i)$. The measure $\mu_0$ is thus the Dirac mass along the canonical divisor $K_{(\Gamma, g)}$ of the genus-augmented graph $(\Gamma, g : v_i \mapsto g(C_i))$, as defined in \cite{ABBR}.
\\The above theorem will follow from the following more general statement about hybrid convergence of measures, in the case where the measures have a "nice" non-collapsed archimedean limit:
\begin{theoB}
Let $X \xrightarrow{\pi} \de^*$ be a projective meromorphic degeneration of compact complex manifolds, and $(\mu_t)_{t \in \de^*}$ a family of probability measures on $X_t$.
\\Assume that there exists a normal model $\X \xrightarrow{\pi} \de$ of $X$, with reduced central fiber $\X_{0} = \sum_{i \in I} D_i$, such that the family of measures $(\mu_t)_{t \in \de^*}$ converges weakly on $\X$ to a measure $\mu_{\infty}$ supported on $\X_{0}$, which does not charge mass on the singular locus of $\X_{0}$.
\\Then the family of measures $\mu_t$ converge weakly on $X^{\emph{hyb}}$ to a finite sum $\mu_0$ of Dirac masses:
$$\mu_0 = \sum_{i \in I} \bigg( \int_{D^o_i} \mu_{\infty} \bigg) \delta_{v_i},$$
where the $v_i$'s are the divisorial valuations corresponding to the $D_i$'s.
\end{theoB}
This result will be proved using an alternative description of the space $X^{\text{hyb}}$ as a projective limit of smaller hybrid spaces associated to snc models of $X$, to be introduced in section 1.2, so that an understanding of the hybrid limiting behaviour of the measures will amount to an understanding of the limit of the measures on the tropicalisations of  $X$ associated to a fairly large class of models of $X$.
\\We refer to section 1.2 for a precise definition of the space $X^{\text{hyb}}$ and of the divisorial valuations involved in the statement of the theorems.
\\Let us briefly describe the organization of the paper. In section 1, we start by recalling the definition of hybrid spaces as introduced by Boucksom-Jonsson, and how their topology can be understood more explicitely in term of Berkovich skeleta.
\\In section 2, we collect various type of information about metric and algebro-geometric convergence of Kähler-Einstein manifolds, that we will use in section 3, where we deduce theorem A from theorem B, before proving theorem B.  We conclude adding some remarks about the non-archimedean Monge-Ampère equation, and discussing some possible extensions to the Kähler-Einstein-Fano case.
\subsubsection*{Conventions and notation}
All complex varieties here are taken with the Euclidean topology.
\\A \emph{pair} $(Z, D)$ consists of a normal complex variety $Z$, together with a Weil $\Q$-divisor $D = \sum_{i \in I} a_i D_i$, $a_i \ge 0$.
\\We refer to \cite{KM} for the standard definitions of log terminal and log canonical singularities for pairs. 
\\On a complex manifold $Z$ of dimension $n$, we set $dd^c = \frac{i}{2\pi} \ddbar$, and the Ricci curvature of a Kähler form $\omega$ is the $(1,1)$-form given locally by $\Ric(\omega) = - dd^c \log \omega^n \in c_1(Z)$.
\subsubsection*{Acknowledgements}
I am grateful to my advisor S. Boucksom for suggesting the question to me and for many helpful comments on the first version of this paper; and to the anonymous referee for his remarks.
\end{section}
\begin{section}{Non-archimedean geometry}
\subsection{Berkovich spaces and the non-archimedean Monge-Ampère operator}
Let $K = \C((t))$ be the field of complex Laurent series, and $X$ be a smooth variety over $K$. 
\\The field $K$ admits an absolute value, given for $f = \sum_{n \in \Z} a_n t^n$ by the formula $\lvert f \rvert = e^{-\text{ord}_0(f)}$, where $\text{ord}_0 (f) = \min \{ n \in \Z / a_n \neq 0 \}$. The valuation ring $\{ \ord_0 \ge 0 \}$ of $K$ is the ring of formal power series $R = \C[[t]]$, which is a discrete valuation ring with residue field $\C$.
\\This absolute value is furthermore \emph{non-archimedean}, in the sense that it satisfies the strong triangle inequality : $\lvert f+g \rvert \le \max{ \{ \lvert f \rvert, \lvert g \rvert \} }$.
\\Using this absolute value, one can now associate to $X$ an analytic space in the sense of Berkovich, which we write as $X^{\text{an}}$. Its underlying topological space satisfies nice properties: it is Hausdorff, locally connected, locally contractible, and is compact if and only $X$ is proper over $K$ - which we will assume from now on.
\\As a set, it consists of real-valued \emph{semi-valuations} on $X_K$, that is, pairs $x = (\xi_x, v_x)$, where $\xi_x$ is a scheme-theoretic point of $X_K$ and $v_x$ is  a valuation on the residue field $K(\xi_x)$ at $\xi_x$, i.e. the field of rational functions of the subvariety $\overline{\{ \xi_x \}}$, extending the valuation $\text{ord}_0$ on $K$. For $x \in X^{\text{an}}$, we will denote $K^{\circ}_x$ the valuation ring $\{ v_x \ge 0 \} \subset K(\xi_x)$ of $x$. 
\\The topology on $X^{\text{an}}$ is now defined as the coarsest making the maps $f \mapsto v_x(f)$ continuous with respect to $x$ for any regular function $f$ on a Zariski open subset of $X$, as well as the forgetful map $X^{\text{an}} \fl X_K$ sending a valuation to the generic point $\xi_x$ corresponding to its field of definition. The space $X^{\text{an}}$ also comes with a structure sheaf which we will not define, as we will not need it here.
\\If $\X$ is a proper model of $X$ over $R$, given a semi-valuation $v_x \in X^{\text{an}}$, the canonical map $\Spec K(\xi_x) \fl X$ extends in a unique way to a map $\Spec K^{\circ}_x \fl \X$, by the valuative criterion of properness. The image of the closed point of $\Spec K^{\circ}_x$ is a scheme-theoretic point of $\X_0$, because of the condition $v_x(t) =1$, and is called the \emph{center} $c_{\X}(x)$ of $x$ on $\X$.
\\Let $L$ be an ample line bundle on $X$. A model of $(X,L)$ is a pair $(\X, \Ld)$ where $\X$ is an snc model of $X$ and $\Ld$ a line bundle on $\X/R$, extending $L$.
\\ Just as the complex analytic case, a continuous metric $\phi$ on $L^{\text{an}}$ is the data, for each local section $s$ of $L$ on a Zariski open subset $U$ of $X$, of a continuous positive function $\lvert s \rvert_{\phi}$ on $U^{\text{an}}$, compatible with restrictions and such that $\lvert fs\rvert = \lvert f \rvert \lvert s \rvert$ for $f \in \gO_X$. Our notations are chosen so that one should think of $\phi$ as the local logarithmic weight of the metric, i.e. $\lvert s \rvert_{\phi} = \lvert s \rvert e^{-\phi}$ in local coordinates, so that the complex analytic first Chern form of the metric locally reads $dd^c \phi$.
\\The most basic examples of such metrics are model metrics:
\begin{defn}
Let $(\X, \Ld)$ be a proper $R$-model of $(X, L)$. The associated model metric $\phi_{\Ld}$ on $L^{\emph{an}}$ is defined by the following property : $\lvert s \rvert_{\phi_{\Ld}} (x) \le 1$ if and only if $s$ extends as a regular section of $\Ld$ at $c_{\X}(v_x)$.
\\It can be computed as follows: if $s_0$ is a trivializing section of $\Ld$ at $c_{\X}(v_x)$, then $\log \lvert s \rvert_{\phi}(x) = -v_x(s/s_0)$.
\end{defn}
If $(\X, \Ld)$ and $(\X', \Ld')$ are two models of $L$, then the model metrics $\phi_{\Ld}$ and $\phi_{\Ld'}$ coincide if and only if there exists a model $\X''$ of $X$ (which can be taken to be snc), dominating both $\X$ and $\X'$, such that the pullbacks of $\Ld$ and $\Ld'$ to $\X''$ coincide.
\\If $\Ld$ is an ample model of $L$, then the above metric $\phi_{\Ld}$ is nothing but $k^{-1}$ times the pullback of the non-archimedean Fubini-Study metric on $\gO(1)$ on $\CP^N$, via an embedding $\phi : \X \hookrightarrow \CP^N_R$ by sections of $\Ld^k$. This metric on the line bundle $\gO(1)$ on $\CP^{N,an}_K$ can be easily described: a section $s \in H^0(U, \gO(1))$ is the data of a linear form $e(x)$ on $l_x$ for each closed point $x \in U$, and the non-archimedean Fubini-Study metric is given by the formula: 
$$\lvert e(x) \rvert_{FS} = \max_{i = 1,...,N} \lvert e(X_i) \rvert_x,$$ 
the $X_i$'s being standard homogeneous coordinates on $\CP^N$, so that the potential for $\theta$ on the standard chart $X_i \neq 0$ is simply $\phi(x) = \log \max_{j \neq i} \frac{\lvert X_j \rvert_x}{\lvert X_i \rvert_x}$.
\\As in the complex analytic case, the metric $\phi$ provides us with a closed (1,1) form $\omega_{\phi}$ written formally $ \omega_{\phi} = dd^c \phi$, which we call the \emph{curvature form} of the metric, which has cohomology class equal to $c_1(L)$ (see \cite{BFJ1} for definition of cohomology classes on $X^{\text{an}}$). If $\phi'$ is another metric on $L$, then $\psi = \phi' -\phi$ is a continuous function on $X^{\text{an}}$, and we have $\omega_{\phi'} = \omega_{\phi} + dd^c \psi$.
\\If $\omega$ is a semi-positive $(1,1)$-form on a complex manifold, the associated Monge-Ampère measure is by definition $\omega^n$, the maximal wedge power of $\omega$. One can thus define the Monge-Ampère measure of a metric $\phi$, by the formula $\MA(\phi) = (dd^c \phi)^n$. This is a positive measure, of total mass $L^n$. 
\\In the non-archimedean setting, it is also possible to associate a measure to a non-archimedean metric on an ample line bundle, as follows:
\begin{defn}
Let $\phi = \phi_{\Ld}$ be the model metric associated to $(\X, \Ld)$, where $\X_0 = \sum_{i \in I} m_i D_i$. Its Monge-Ampère measure is defined by the following formula:
$$ \MA(\phi_{\Ld}) = \sum_{i \in I} m_i(\Ld^n \cdot D_i) \delta_{v_i}.$$
\end{defn}
Even if this may not be transparent at first sight, the definition is designed to mimic the real Monge-Ampère operator on piecewise affine convex functions, and yields on operator that is expected to be related to the latter operator on skeleta $\Sk(\X)$ of snc models of $X$. 
\\The Monge-Ampère can then be extended to semi-positive metrics on $X^{\text{an}}$, using density of model metrics (and even those of metrics of \emph{finite energy}, see \cite{BFJ1}). This allows one to define the Monge-Ampère measure $\MA(\phi)$, for any semi-positive metric $\phi$ on $L$.
\begin{rem}
It is actually possible to give a meaning to semi-positive currents of arbitrary degree - such as $dd^c \phi$ - and under some assumptions, their various wedge powers on Berkovich spaces over $K$, using the theory developed in \cite{CLD}, which provides another construction of the Monge-Ampère measure associated to a continuous semi-positive metric.
\end{rem}
The following analog of the Calabi-Yau theorem holds in this non-archimedean setting:
\begin{theo}{\cite[theo. A]{BFJ1}, \cite[theo. D]{BGJKM}}
\\Let $X \fl \de^*$ be an meromorphic degeneration of projective complex manifolds, polarized by a relatively ample line bundle $L$.
\\Let $\mu$ be a positive Radon measure on the Berkovich analytification $X^{\emph{an}}$, supported on the skeleton $\emph{Sk}(\X)$ of some snc model $\X$ of $X$ (see section 1.2 for the definition), and with total mass $\mu(X^{\emph{an}}) = L^n$.
\\Then there exists a unique semi-positive metric on $L^{\emph{an}}$, satisfying the non-archimedean Monge-Ampère equation on $X^{\emph{an}}$:
$$ \MA(\phi) = \mu.$$
\end{theo}
This theorem was first proved when $X$ is defined over an algebraic curve in \cite{BFJ1}, and then extended to varieties over non-archimedean fields of residual characteristic zero in \cite{BGJKM}.
\\The main motivation for such a result comes from a global motto that suggests that when $X$ comes from a degeneration of complex manifolds, the non-archimedean space $X^{\text{an}}$ should capture asymptotic data associated to $X$ near $t=0$: here the above Monge-Ampère equation should be thought as an asymptotic version of the Monge-Ampère equation on the complex analytic fibers. This can be made more rigorous using the formalism of hybrid spaces, which we will now explain.
\subsection{Hybrid spaces and skeleta}
Let us explain how the Berkovich construction can be generalized to schemes of finite type over a Banach ring, as this will allows us to define the hybrid spaces we are interested in. We refer to \cite[App. A]{BJ} and the references therein. 
\\Recall that a Banach ring $A$ is a commutative ring equipped with a submultiplicative norm $\lVert \cdot \rVert$ such that $A$ is complete with respect to $\lVert \cdot \rVert$. 
\\First, the Berkovich spectrum $\M(A)$ is the set whose points $x \in \M(A)$ are multiplicative semi-norms $\lvert \cdot \rvert_x : A \fl \R_{\ge 0}$ which satisfy $\lvert \cdot \rvert_x \le \lVert \cdot \rVert$. It is equipped with the topology of pointwise convergence on $A$, which makes it into a non-empty Hausdorff compact topological space, and with a map $\lvert \cdot \rvert_x \mapsto \mathfrak{p}_x = \Ker (\lvert \cdot \rvert_x)$ to $\Spec(A)$ which is continuous.
\\For example, assume that $A$ is a complex Banach algebra: the classical Gelfand-Mazur theorem implies that $\M(A)$ is nothing but the set of maximal ideals of $A$.
\\Now let $B$ be a finitely generated $A$-algebra. The analytification $Y^{\text{An}}$ of $Y =\Sp(B)$ is defined as the set of multiplicative semi-norms on $B$ whose restriction to $A$ belong to $\M(A)$. It is endowed with the coarsest topology making the maps $x \mapsto \lvert f \rvert_x$ continuous, for $f \in B$, and comes with a structure morphism $Y \fl \M(A)$ sending a semi-norm to its restriction to $A$.
\\If $X$ is a scheme of finite type over $A$, then one can glue the analytifications of affine charts of $X$ in order to obtain an \emph{analytification functor} $X \mapsto X^{\text{An}}$.
\begin{defn}
Let $\C^{\emph{hyb}}$ be the Banach algebra obtained by equipping the standard $\C$ with the norm $\lVert \cdot \rVert_{\emph{hyb}}$, defined for $z\neq 0$ by:
$$ \lVert z \rVert_{\emph{hyb}} = \max \{ 1 , \lvert z \rvert \},$$
where $\lvert z \rvert$ is the usual modulus of $z$.
\end{defn}
One can show that the elements of $\M(\C^{\text{hyb}})$ are of the form $\lvert \cdot \rvert^{\lambda}$, for $\lambda \in [0,1]$, where $\lvert \cdot \rvert^0 = \lvert \cdot \rvert_0$ denotes the \emph{trivial} absolute value on $\C$ - that is, the absolute value such that $\lvert z \rvert_0 =1$ for all $z \neq 0$, which is a \emph{non-archimedean} absolute value. This yields a homeomorphism $\lambda : \M(\C^{\text{hyb}}) \xrightarrow{\sim} [0,1]$.
\\Thus, if $Z$ is a scheme of finite type over $\C$, its analytification with respect to $\lvert \cdot \rvert_{\text{hyb}}$, which we denote by $Z^{\text{hyb}_{\C}}$, comes with a structure morphism $p : Z^{\text{hyb}_{\C}} \fl [0,1]$. If $X = \Spec(A)$ is affine, the fiber over $\lambda \neq 0$ is by definition of $p$ the set of semi-norms extending the absolute value $\lvert \cdot \rvert^{\lambda}$ on $\C$, so that by rescaling and by the Gelfand-Mazur theorem, this is easily seen to be homeomorphic to the holomorphic analytification $Z^{\text{hol}}$ of $Z$. One can actually show that we have a $p$-homeomorphism : 
$$p^{-1}((0,1]) \xrightarrow{\sim} (0,1] \times Z^{\text{hol}}.$$
\\On the other hand, the fiber $p^{-1}(0)$ consists of the semi-norms extending the trivial absolute value on $\C$, so that this is homeomorphic to the analytification $Z^{\text{an}}_0$ of $Z$ with respect to the trivial absolute value on $\C$.
\\Thus, the space $Z^{\text{hyb}_{\C}} $ provides a natural way to degenerate (topologically) the complex manifold $Z^{\text{hol}}$ to the non-archimedean analytic space $Z^{\text{an}}_0$.
\\Let us now perform a similar construction for schemes $X/K$ of finite type over $K = \C((t))$.
\\We fix a radius $r \in (0,1)$, and consider the following Banach ring:
$$A_r = \{ f = \sum_{n \in \Z} a_n t^n \in \C((t))  \; / \; \lVert f \rVert_{\text{hyb}} = \sum_{n} \lVert a_n \rVert_{\text{hyb}} \; r^n < \infty \}.$$
The purpose of the above Banach ring is to provide a presentation of the complex disk as an affine non-archimedean analytic space (at least topologically):
\begin{lem}{(\cite[prop. A.4]{BJ})}
\\The map $z \fl \lvert \cdot \rvert_z$ defined by:
\[ \lvert f \rvert_z = \left\{ \begin{array}{ll} r^{\ord_0(f)} & \mbox{if $z =0$},\\
\lvert f \rvert_z = r^{\log \lvert f(z) \rvert/ \log \rvert z \rvert} & \mbox{if $z \neq 0$}
\end{array} \right. \]
\\induces a homeomorphism from $\bar{\de}_r$ to $\M(A_r)$.
\end{lem}
Arguing as above, one can show:
\begin{prop}
Let $X$ be a scheme of finite type over $A_r$. 
\\The associated analytic space $\pi : X^{\emph{An}} \fl \M(A_r) \simeq \bar{\de}_r$ satisfies:
$$\pi^{-1}(0) \simeq X^{\emph{an}} \; \text{and} \; \pi^{-1} (\de_r^*) \simeq X^{\emph{hol}},$$
where $\simeq$ denotes homeomorphisms.
\end{prop}
Now, if $X \xrightarrow{\pi} \de^*$ is a projective meromorphic family of complex manifolds, it is defined by a finite number of polynomial equations over $A_r$, for every $r<1$, and one can define $X^{\text{hyb}}_r$ as the analytification $X^{\text{An}}_r$ of $X$ viewed as a scheme of finite type of $A_r$. Taking the union of those for all $r<1$ yields the following:
\begin{defn}
Let $X \xrightarrow{\pi} \de^*$ be a projective degeneration of complex manifolds. The associated hybrid space $X^{\emph{hyb}}$ is the union of analytic spaces $X^{\emph{An}}_r$ associated to $X$ for all $0<r<1$, as defined above.
\\It comes with a structure map $p : X^{\emph{hyb}} \fl \de$, whose restriction to $\de^*$ coïncides with $\pi$, and such that $p^{-1}(0) \simeq X^{\emph{an}}$.
\end{defn}
In order to study convergence of measures on the hybrid space $X^{\text{hyb}}$ associated to a projective degeneration of complex manifolds, we will now describe how the topology of $X^{\text{hyb}}$ can be understood in terms of "smaller" and more tractable hybrid spaces $\X^{\text{hyb}}$ associated to simple normal crossing models $\X/ \de$ of $X$.
\\A model $\X$ of $X$ is \emph{simple normal crossing} (snc for short) if it has smooth total space, and is such that the central fiber $\X_0$ has simple normal crossings support, i.e.  $\X_0 = \sum_{i \in I} a_i D_i$, where the $D_i$ are smooth prime divisors, intersecting each other transversally.
\\Under this simple normal crossings assumption, for all $J \subset I$, each $D_J = \cap_{j \in J} D_j $ is smooth of dimension $n+1 - \lvert J \rvert$, and a connected component $Y$ of such a $D_J$ will be called a \emph{stratum} of $\X_0$.
\\Let now $(Z, D)$ be a pair, consisting of a complex manifold $Z$ and a simple normal crossing divisor $D = \sum_{i \in I} m_i D_i$.
\\We recall the construction of Boucksom-Jonsson of the hybrid space $Z^{\text{hyb}_D}$ associated to the pair $(Z, D)$, mainly following \cite[sec. 2]{BJ}. If the manifold $Z$ is compact, this provides a compactification of $Z \backslash D$ by adding as a boundary the dual complex of $D$, a cell complex encoding the combinatorics of the intersections of the components of $D$. Faces of this complex are in one-to-one correspondence with \emph{strata} of $D$, that is, connected components $Y$ of the non-empty intersections $D_J = \cap_{j \in J} D_j$ for $J \subset I$. We recall the precise definition:
\begin{defn}
Let $(Z, D)$ be a simple normal crossing pair. 
To each stratum $Y$ of $D$ which is a connected component of $D_J$, we associate a simplex:
$$ \sigma_Y = \{ w \in \R_{\ge0}^{J} / \sum_{j \in J} m_j w_j =1 \}.$$
We define the cell complex $\cD(D)$ by the following incidence relations: $\sigma_Y$ is a face of $\sigma_{Y'}$ if and only if $Y' \subset Y$.
\end{defn}
For instance, if $Z$ is a proper surface, and $D$ a nodal curve without self-intersection on $Z$, then $\cD(D)$ is nothing but the dual graph of the curve $D$.
\\As a set, the hybrid space $Z^{\text{hyb}_D}$ is by definition:
$$ Z^{{\text{hyb}_D}} = Z \backslash D \sqcup \cD(D).$$
\\The topology on $Z^{\text{hyb}_D}$ is defined using the local $\Log$ maps, which are given as follows. These maps are defined on every open chart $(U, z)$, with $z \in \de^n$, $z_i$ a local equation for $D_i$, satisfying the following property: if $J$ is the set of $j \in I$ such that $U$ meets $D_j$, then $U \cap D$ meets only one component $Y_U \subset D_J$ (those charts are called \emph{adapted} in \cite{BJ}).  Given such an open chart, we define a continuous map:
$$ \Log_U : U \backslash D \fl \sigma_U,$$
$$ z \longmapsto \bigg(\frac{\log \lvert z_j \rvert}{ \log \lvert f_U \rvert}\bigg)_{j \in J},$$
where $f_U = \prod_{j \in J} z_j^{m_j}$ is a local equation for $D$ on $U$.
\\We then extend this map to $U^{\text{hyb}_D} := U \backslash D \sqcup \sigma_Y$ by setting $\Log_U = \Id_{\sigma}$ on $\sigma_U$.
This construction can then be globalized (although non-canonically), using a partition of unity argument:
\begin{prop}{(\cite[prop. 2.1]{BJ})}
\\There exists an open neighbourhood $V$ of $D$ in $Z$, and a map $\Log_V : V^{\emph{hyb}_D} \fl \cD(D)$, satifying the following property: for every chart $(U, z)$ as above, and $U \subset V$, we have $\Log_V (U \backslash D) \subset \sigma_U$ and:
$$ \Log_V = \Log_U + O((\log \lvert f_U \rvert^{-1})^{-1})$$
uniformly on compacts subsets of $U$.
\end{prop}
We are now ready to define the topology on $Z^{\text{hyb}_D}$:
\begin{defn}
The hybrid topology on $Z \backslash D \sqcup \cD(D)$ is the coarsest topology such that:
\\i) $Z \backslash D$ is open inside $Z^{\emph{hyb}_D}$;
\\ ii) for every open neighbourhood $V$ of $D$ inside $Z$, $V \backslash D \sqcup \cD(D)$ is open in $Z^{\emph{hyb}_D}$;
\\iii) $\Log_V : V^{\emph{hyb}_D} \fl \cD(D)$ is continuous.
\end{defn}
It follows from the above proposition that the topology on $Z^{\text{hyb}_D}$ is independent of the choice of $V$ and map $\Log_V$.
\begin{ex}
Let $\mathcal{D}= G/K$ be a Hermitian symmetric domain, and $\Gamma \subset O(\mathcal{D})$ an arithmetic subgroup.
\\It is a classical problem in representation theory to study compactifications of quotients of the form $X =\Gamma \backslash \mathcal{D}$. 
\\For instance, Satake \cite{Sa} associates to representations $\tau$ of $G$ topological compactifications $\bar{X}^{\tau}$ of $X$, which generalize the classical Baily-Borel compactification.
\\The boundary of such compactifications have typically very large codimension; a way to produce compactifications with divisorial boundary is via so-called \emph{toroidal compactifications}, which have the defect of being non-canonical and depend on heavy combinatorial data.
\\It turns out that those two constructions are related by the above hybrid construction: for any two toroïdal compactifications $\bar{X}_1$, $\bar{X}_2$ of $X$, the associated hybrid spaces $\bar{X}_1^{\emph{hyb}} \xrightarrow{\sim} \bar{X}_2^{\emph{hyb}}$ are canonically homeomorphic (extending the identity of $X$), and actually coïncide with the Satake compactification $\bar{X}^{\tau_{ad}}$ associated to the adjoint representation of $G$, by \cite[theo. I]{OO}
\end{ex}
Going back to our original set-up, where $X \xrightarrow{\pi} \de^*$ is meromorphic degeneration of Kähler-Einstein manifolds, the above construction allows us to associate to each snc model $\X / \de$ of $X$ a hybrid space which we denote by $\X^{\text{hyb}} \xrightarrow{\pi} \de$, associated to the pair $(\X, \X_0)$, with central fiber $\pi^{-1}(0) = \cD(\X_0)$.
\\Those hybrid spaces now carry information about the combinatorics of models of $X$, and can be seen as form of tropicalizations of the snc models of $X$. 
\\Let us now explain how the "big" hybrid space $X^{\text{hyb}}$ can be recovered using the collection of all $\X^{\text{hyb}}$. To that end, one defines maps of dual complexes $r_{\X' \X} : \cD(\X'_0) \fl \cD(\X_0)$ for every ordered pair $h : \X' \fl \X$ of snc models, as follows: given a face $\sigma_Y'$ of $\cD(\X')$, there exists a unique minimal stratum $Y$ of $\X_0$ such that $h(Y') \subset Y$.
\\Writing $J$ (resp. $J'$) for the set of indices $j$ such that $Y \subset D_J$ (resp. $Y' \subset D'_j$), we write for every $i \in J$ the pullback $h^*D_i = \sum_{j \in J'} a_{ij} D'_j$. The map $r_{\X' \X}$ is now defined by mapping the simplex $\sigma_{Y'}$ to $\sigma_Y$, according to the formula:
$$r_{\X' \X} ((w'_j)_{j \in J'}) = (\sum_{j \in J'} a_{ij} w'_j)_{i \in J}.$$
One can check that for any ordered triple $\X'' \fl \X' \fl \X$ of snc models of $X$, one has the transitivity property $r_{\X'' \X} = r_{\X'' \X'} \circ r_{\X' \X}$. 
\\Extending the maps $r_{\X' \X}$ to the hybrid spaces $\X^{\text{hyb}}$ by setting $r = \Id$ on $X$, we make the partially ordered set of snc models of $X$ into a projective system, and one can construct a homeomorphism:
$$X^{\text{hyb}} \xrightarrow{\sim} \varprojlim_{\X} \X^{\text{hyb}},$$
the right hand side being equipped with the inverse limit topology.
\\We will now describe the construction of the above homeomorphism. Given any snc model $\X$ of $X$ over $R $, there is a natural embedding $i_{\X}$ of the dual complex $\cD(\X)$ into $X^{\text{an}}$, given as follows. 
\\The vertices $v_i$ of $\cD(\X)$ are in one-to-one correspondence with irreducible components $D_i$ of the central fiber $\X_0 = \sum_{i \in I} m_i D_i$, so that we set $i_{\X}(v_{i}) = v_{D_i} := m^{-1}_i \ord_{D_i}$, where the valuation $\ord_{D_i}$ is simply the one associating to a meromorphic function $f \in K(X) \simeq K(\X)$ its vanishing order along $D_i$ - the normalisation by $m^{-1}_i$ ensuring that $v_{D_i}(t) = 1$. A valuation given in this way, for some snc model $\X$ of $X$, is called \emph{divisorial}.
\\One can now somehow interpolate between those divisorial valuations using \emph{quasi-monomial} valuations, in order to embed $\cD(\X)$ into $X^{\text{an}}$:
\begin{prop}{(\cite[prop. 2.4.6]{MN}).}
\\Let $\X$ be an snc model of $\X$, with central fiber $\X_0 = \sum_{i \in I} m_i D_i$.
\\Let $J \subset I$ such that $D_J = \cap_{j \in J} D_j$ is non-empty, and $Y$ a connected component of $D_J$, with generic point $\eta = \eta_Y$. 
\\We furthermore fix a local equation $z_j \in \gO_{\X,\eta}$ for $D_j$, for any $j \in J$.
\\For any $w \in \sigma_Y = \{ w \in \R^J_{\ge 0} / \sum_{j \in J} m_j w_j =1 \}$, there exists a unique valuation: 
$$v_w : \gO_{\X,\eta} \fl \R_{\ge_0} \cup \{ + \infty\}$$
 such that for every $f \in \gO_{\X,\eta}$, with expansion $f = \sum_{\beta \in \N^{J}} c_{\beta} z^{\beta}$ (with $c_{\beta}$ either zero or unit), we have:
$$v_w(f) = \min \{ (w \cdot \beta) / \beta \in \N^{J},  c_{\beta} \neq 0 \},$$
where $( \; \cdot \; )$ is the usual scalar product on $\R^J$.
\end{prop}
The above valuation is called the quasi-monomial valuation associated to the data $(Y, w)$. We now define $i_{\X}$ on $\sigma_Y$ by the formula $i_{\X}(w) = v_w$, which gives us a well-defined continuous injective map from $\cD(\X)$ to $X^{\text{an}}$. 
\begin{defn}
We call the image of $\cD(\X)$ by $i_{\X}$ the \emph{skeleton} of $\X$, written as $\emph{Sk}(\X) \subset X^{\emph{an}}$.
\end{defn}
By compactness of $\cD(\X)$, $i_{\X}$ induces a homeomorphism between $\cD(\X)$ and $\Sk(\X)$, so that we will sometimes abusively identify $\cD(\X)$ with $\Sk(\X)$.
\\We can now define a retraction for the inclusion $\Sk(\X) \subset X^{\text{an}}$ as follows: for a semi-valuation $v$ on $X$, there exists a minimal stratum $Y$ of $\X_0$ such that the center $c_{\X}(v)$ of $v$ is contained in $Y$. We then associate to $v$ the quasi-monomial valuation $\rho_{\X}(v)$ associated to the data $(Y, w)$, with $w_j = v(D_j)$. This should be seen as a monomial approximation of the valuation $v$, with respect to the model $\X$.
\\This retraction map can be used to study the homotopy type of the space $X^{\text{an}}$:
\begin{theo}{(\cite{Be}, \cite{Th}).}
\\The map $\rho_{\X} : X^{\emph{an}} \fl \emph{Sk}(\X)$ is continuous and restricts to the identity on $\emph{Sk}(\X)$.
\\It furthermore realizes $\emph{Sk}(\X)$ as a strong deformation retract of $X^{\text{an}}$.
\end{theo}
\begin{rem}
Under the identification $\emph{Sk}(\X) \simeq \cD(\X)$, it is not difficult to show that the map $r_{\X' \X} : \cD(\X') \fl \cD(\X)$ for an ordered pair of models $h : \X' \fl \X$ is nothing but the restriction of $\rho_{\X}$ to $\emph{Sk}(\X')$.
\end{rem}
This shows that up to homotopy, the topological type of $X^{\text{an}}$ is controlled by the one of the complexes $\cD(\X)$. Furthermore, this shows that the homotopy type of $\cD(\X)$ is in fact independent of the choice of snc model $\X$.
\\The following result now gives a rather concrete description of the Berkovich space $X^{\text{an}}$ (or at least, its underlying topological space):
\begin{theo}{(\cite[theo. 10]{KS}).}
\\The natural map $\rho : X^{\emph{an}} \fl \varprojlim_{\X} \cD(\X)$ is a homeomorphism.
\end{theo}
This in particular implies our claim that:
$$X^{\text{hyb}} \xrightarrow{\sim} \varprojlim_{\X} \X^{\text{hyb}}.$$
\subsection{Degenerations of Calabi-Yau manifolds}
In the case where the variety $X/K$ has semi-ample (i.e., eventually base-point free) relative canonical divisor $K_X$, one can actually define a "smallest" skeleton $\Sk(X)$, contained in any skeleton $\Sk(\X)$ for an snc model $\X$.
\\Without giving the precise definition (see for instance \cite{MN}), we rather explain how $\Sk(X)$ can be computed explicitely, starting from a simple normal crossing model $\X / R$, with central fiber $\X_0 = \sum_{i \in I} m_i D_i$.
\\For each holomorphic $m$-pluriform $\Omega \in H^0(\X, mK_{\X/R})$ on $\X$, we attach weights to the vertices $v_i$ of $\Sk(\X)$, according to the formula $w_i = \frac{\ord_{D_i} (\Omega) +m}{m_i}$, and define $\Sk_{\Omega}(X)$ to be the union of closed faces $\sigma$ of $\Sk(\X)$ such that $w_i = \min_i w_{i}$ for every vertex $w_i \in \sigma$. This is indeed independent of the choice of snc model $\X$, as follows from \cite{MN}.
\\Setting $\Sk(X) = \bigcup_{\Omega} \Sk_{\Omega}(X)$, we obtain a \emph{canonical} skeleton inside $X^{\text{an}}$. This is the essential skeleton of $X$, or Kontsevich-Soibelman skeleton.
\\In the strictly Calabi-Yau case (that is, $K_X = \gO_X$ and $h^{i,0}(X) =0$ for $0<i<n$), $\Sk(X) = \Sk_{\Omega}(X)$ for any n-form $\Omega$ trivializing $K_X$, and if it is of maximal dimension, $\Sk(X)$ satisfies nice topological properties: it is a rational homology sphere, and a topological manifold in codimension one, cf. \cite{NX}. It is actually expected to be homeomorphic to the n-sphere $\mathbb{S}^n$ in this case, see \cite{KX} for partial results towards this.
\\We can now state the main result of \cite{BJ}:
\begin{theo}
Let $(X, L) \xrightarrow{\pi} \de^*$ be a relatively polarized degeneration of n-dimensional Calabi-Yau manifolds, and let $\nu_t = i^{n^2} \Omega_t \wedge \bar{\Omega}_t$.
\\The total mass $\nu_t(X_t)$ obeys the following asymptotic at $t=0$:
$$ \nu_t(X_t) \sim C \lvert t \rvert^{2 \kappa} (\log \lvert t \rvert^{-1})^d,$$
with $\kappa \in \Q$ and $d \in \{0,...,n\}$ is the dimension of $\emph{Sk}(X)$.
\\Moreover, the family of rescaled measures (or Kähler-Einstein measures) $\mu_t = \frac{\nu_t}{\nu_t(X_t)}$ converge weakly on $X^{\emph{hyb}}$ to a Lebesgue-type measure $\mu_0$ supported on the essential skeleton $\emph{Sk}(X)$.
\end{theo}
The above result is of particular interest in the case where the family is \emph{maximally degenerate} - that is, when $d=n$ in the above formula. In that case, the essential skeleton $\Sk(X)$ is expected to be the base of the SYZ fibration predicted by mirror symmetry. In particular, Kontsevich and Soibelman conjectured that the renormalized Kähler Ricci-flat metrics with diameter one on $X_t$ should collapse to $\Sk(X)$, endowed with a real Monge-Ampère metric, i.e. a metric $g_{ij}$ satisfying $\det(g_{ij}) =1$ in local $\Z$-affine coordinates (see \cite[Conj. 5.1]{KS} for a precise statement).
\\In the maximally degenerate case, Boucksom-Jonsson furthermore prove that after a finite base change, the limiting measure $\mu_0$ is a positive multiple of the Lebesgue measure on $\Sk(X)$, so that their result provide interesting evidence towards the Kontsevich-Soibelman conjecture.
\\Moreover, in view of this result, one would naturally be tempted to expect some continuity for the Monge-Ampère operators : if the family of measures $(\mu_t)_{t \in \de^*}$ converges on the hybrid space to a Radon measure $\mu$ on $X^{\text{an}}$ which satisfies the hypotheses of theorem 1.1, then the solutions $\omega_t \in c_1(L_t)$ of the Monge-Ampère equation $\omega_t^n = \mu_t$ converge in some sense as metrics on $L$ to the solution $\phi$ of the non-archimedean Monge-Ampère equation $\MA( \phi) = \mu$.
\\To define of hybrid convergence of metrics, we will reduce the problem to convergence of potentials, by fixing a natural class of metrics which have a non-archimedean counterpart, and making them hybrid-converge by definition. 
\\Those will be the Fubini-Study metrics: for any model $\Ld$ of our line bundle $L$, we may write $\Ld = \Ld_1 \otimes \Ld_2^{-1}$ as a difference of two ample line bundles $\Ld_1$ and $\Ld_2$, which yield embeddings $\iota_i : \X \hookrightarrow \C\CP^{N_i} \times \de$ by sections of $\Ld_i^{k_i}$.
\\This allows us to define a family of metrics on $L$, $\phi_t = k_1^{-1} \iota_1^* \phi_{FS,1} - k_2^{-1} \iota_2^* \phi_{FS,2}$, and similarly a model metric $\phi = \phi_{\Ld_1} - \phi_{\Ld_2}$ on $X^{\text{an}}$.
\begin{defn}
Let $\X /R$ be an snc model of $X$, with ample model $\Ld$ of $L$.
\\We say that the family of hermitian metrics on $L$ $\phi'_t = \phi_t + \psi_t$, with $\sup_{X_t} \psi_t =0$ converge \emph{strongly} to $\phi' = \phi + \psi$ with $\sup_{X^{\emph{an}}} \psi =0$ if the function:
$$ \Psi : X^{\emph{hyb}} \fl \R,$$
$$ z \mapsto \psi_t(z) \; \emph{if} \; z \in X_t,$$
$$ z \mapsto \psi(z) \; \emph{if} \; z \in X^{\emph{an}},$$
is continuous on $X^{\emph{hyb}}$.
\end{defn}
The above definition obviously implies that the Fubini-Study metrics $\phi_t$ converge strongly to the model metric $\phi$ on $X^{\text{an}}$ - on the level of local potentials this means that on the hybrid space, the Fubini-Study usual $\text{L}^2$-norm on standard coordinates for $\CP^N$ degenerates to the $\text{L}^{\infty}$-norm on these coordinates when $t \rightarrow 0$. This can be seen easily using for instance \cite[thm. 1.2]{Fav}.
\end{section}
\section{Metric and algebraic convergence}
In this subsection, $X \xrightarrow{\pi} \de^*$ will denote a meromorphic degeneration of canonically polarized manifolds. It follows from the seminal work of Aubin and Yau (\cite{Au}, \cite{Y}) that every fiber $X_t$ admits a unique negatively curved Kähler-Einstein metric $\omega \in -c_1(X_t)$, satisfying the equation $\Ric(\omega_t) = - \omega_t$. We will now explain how to understand the asymptotic behaviour of the $\omega_t$'s as $t \rightarrow 0$.
\\In this case, the Minimal Model Program provides us with a \emph{unique} canonical model $\X_c$ of $X$ over the disk, at the cost of going out of the class of simple normal crossing models, and allowing some slightly worse singularities. The appropriate class of varieties for the central fiber is a higher-dimensional analog of the stable curves, the correct notion being that of \emph{semi-log canonical models}.
\\If $\X$ is a normal model of $X$, saying that $\X_0$ is semi-log canonical (see for instance \cite{Kol}) is a condition on the singularities of the normalisation of $\X_0$, which can be seen as a mild generalization of the simple normal crossing condition; in particular we require $\X_0$ to be simple normal crossing in codimension 1. More precisely, the normalization morphism $\nu : \X^{\nu}_0 \fl \X_0$ is required to yield a disjoint union $\X_0^{\nu} = \sqcup_{i \in I} (\tilde{D}_i , C_i)$ of log canonical pairs, $C_i$ being the restriction of the \emph{conductor} $C$ of $\nu$ to $\tilde{D}_i$. This is a Weil divisor on $\X_0^{\nu}$, whose support is precisely the locus where $\nu$ fails to be an isomorphism, and which is simply given here by the inverse image by $\nu$ of the codimension one nodes of $\X_0$. It furthermore satisfies the formula : $\nu^*K_{\X_0} = K_{\X^{\nu}_0} + C$ (note that the canonical divisor of a semi-log canonical variety is assumed to be $\Q$-Cartier).
\\A semi-log canonical model (or \emph{stable variety}) is now by definition a proper semi-log canonical variety, with ample canonical divisor. Thus, as we mentioned above, one-dimensional semi-log canonical models are nothing but Deligne-Mumford's stable curves. 
\\The compactness theorem for stable varieties of higher dimension is now as follows:
\begin{theo}{(\cite{BCHM}, \cite{KNX}).}
\\Let $X \fl \de^*$ be an algebraic degeneration of canonically polarized manifolds. 
\\There exists (possibly after a finite base change) a unique \emph{canonical} model $\X_c / \de$ of $X$, satisfying the following properties:\\
i) the total space $\X_c$ has at worst canonical singularities, while the central fiber $\X_{c,0}$ is reduced and has at worst semi-log canonical singularities;\\
ii) the relative canonical divisor $K_{\X_c/ \de}$ is relatively ample.
\end{theo}
The canonical model is constructed as follows: starting from a semi-stable model $\X / \de$ of $X$, we set $\X_c = \Proj_{\de} \bigoplus_{k \ge 0} H^0 (\X, K^k_{\X/ \de})$, the main difficulty being to prove finite-generatedness of the relative canonical algebra. This is established in \cite{BCHM} when $X / \de^*$ is algebraic, and extented to families over the disk in \cite{KNX}.
\begin{rem}
If $\X$ is any semi-stable model of $X$, then the natural rational map $h : \X \dashrightarrow \X_c$ is in fact a \emph{rational contraction} - this means that its inverse does not contract any divisors.
\end{rem}
The complete understanding of the Gromov-Hausdorff convergence of the fibers $(X_t, g_t)$, is due to J. Song \cite{S} (whose results were further improved very recently in \cite{SSW}). The crucial, first step, is to show that there exists on the central fiber $\X_{c,0}$ of the canonical model of $X$ a unique Kähler-Einstein current $\omega_{KE}$, and to derive some geometric estimates on the singularities of this current. The current $\omega_{KE}$ on the stable variety $\X_{c,0}$ was first constructed by Berman-Guenancia \cite{BG} using a variational method, while it is reconstructed in \cite{S} using the techniques of \cite{EGZ}, \cite{Kolo}, in order to obtain some stronger control on its singularities:
\begin{theo}{(\cite[thm. 1.1]{S}).}
\\Let $\X_c \fl \de$ be the canonical model of $X$, with semi-log canonical central fiber $\X_{c,0}$.
\\There exists a unique Kähler current $\omega_{KE} \in - c_1(\X_{c,0})$ on $\X_{c,0}$, satisfying the following properties:
\\i) $\omega_{KE}$ is smooth and satisfies the Kähler-Einstein equation on the regular locus of $\X_{c,0}$;
\\ii) $\omega_{KE}$ has locally bounded potentials away from the locus where $\X_{c,0}$ is not log terminal;
\\iii) $\omega_{KE}^n$ does not charge mass on the singularities of $\X_{c,0}$, and $\int_{\X_{c,0}} \omega_{KE}^n = [K_{\X_{c,0}}]^n$.
\end{theo}
\begin{rem}
The fact that the above Kähler-Einstein current on $\X_{c,0}$ matches the one constructed in \cite{BG}, follows from the uniqueness statement in \cite[thm. A]{BG}. Moreover, the construction of \cite{BG} implies that $\int_{D_i} \omega_{KE}^n = (K_{\X_c}^n \cdot D_i)$. 
\\Indeed, if $\nu : \X^{\nu}_{c,0} \fl \X_{c,0}$ denotes the normalization morphism, where $\X_{c,0}^{\nu} = \sqcup_{i \in I} \tilde{D}_i$, then the Kähler-Einstein metric $\omega_{KE}$ is obtained by descending the (singular) Kähler-Einstein metrics $\omega_i \in c_1(K_{\tilde{D}_i} + C_i)$ on the log canonical pairs $(\tilde{D_i}, C_i)$, $C_i$ being the restriction of the conductor $C$ of $\nu$ to $\tilde{D}_i$.
\\Thus by construction, the mass $\int_{D_i} \omega_{KE}^n$ equals the intersection number $(K_{\tilde{D}_i} + C_i)^n = (\nu^* K_{\X_{c,0}})^n$, the last intersection number being computed on $\tilde{D}_i$. Applying the projection formula, this is equal to the intersection number $ K_{\X_{c,0}}^n \cdot D_i =K_{\X_c}^n \cdot D_i$, by adjunction and principality of $\X_{c,0}$.
\end{rem}
Let us now fix a $k>0$ such that $kK_{\X_c/\de}$ is relatively very ample, and a relative embedding $\iota : \X_c \hookrightarrow \CP^N \times \de$ of the canonical model inside projective space by sections of $k K_{\X_c/ \de}$. We write $\theta_t = k^{-1} \iota_t^*\omega_{FS}$ the smooth family of reference metrics induced on $\X_c$ via pull-back of the Fubini-Study metric on $\CP^N$, and $\nu_t = \theta_t^n$ for the associated reference measures.
\\This allows us to write $\omega_{t} = \theta_t + dd^c \psi_t$ and $\omega_{KE} = \theta_0 + dd^c \psi_0$, with $\psi_t \in \mathcal{C}^{\infty}(X_t)$ and $\psi_0 \in \PSH(\X_{c,0}, \chi)$. We normalize $\psi_t$ by imposing the condition $\int_{X_t} e^{\psi_t} \theta_t^n =1$, and similarly for $\psi_0$, so that this guarantees uniqueness of the potentials.
\\In order to apply Cheeger-Colding theory to the Kähler-Einstein metrics on $X_t$, Song derives uniform estimates on volumes of small balls, which are obtained via comparison lemmas for volume forms. The estimates we will need in the sequel are the following :
\begin{prop}{(\cite[lem. 4.5, cor. 4.1]{S})}
\\We have a uniform bound: $\sup_{X_t} \psi_t \le C$.
\\Furthermore, for any compact $K \subset \X_c$ disjoint from $\Sing(\X_{c,0})$, the following bounds hold:
\begin{itemize} 
\item $\inf_{X_t} \psi_t \ge - C_K$, 
\item $\lVert \psi_t \rVert_{\mathcal{C}^k(K)} \le C_{K,k}$.
\end{itemize}
This implies smooth convergence of the Kähler-Einstein metrics $\omega_{t}$ to $\omega_{KE}$ on $\X_{c,0} \backslash \Sing(\X_{c,0})$ in the following sense:
\\for any point $p \in \X_{c,0} \backslash \Sing(\X_{c,0})$, and any choice of neighbourhood $\U$ of $p$ such that the $U_t = \U \cap X_t$ are all biholomorphic to $U_0$, and such that the $\theta_t$'s are uniformly equivalent, the $\psi_t$ converge in the $\mathcal{C}^{\infty}$ sense to $\psi_0$.
\end{prop}
\begin{rem}
Even if we will not need it here, one can show that the previous theorem combined with a uniform non-collapsing condition implies pointed Gromov-Hausdorff convergence of $X_t$ to a compact metric space, whose regular part (in the Cheeger-Colding sense) is precisely $(\X_{c,0} \backslash \Sing(\X_{c,0}), \omega_{KE})$.
\\The behaviour of the metrics in the region where the metric collapses is also well-understood, under the technical assumption that the canonical model is semi-stable, see \cite{Zh}.
\end{rem}
We conclude with a lemma about the singularities of the potential $\psi_0$, this will be useful to us when discussing the behaviour of the Kähler-Einstein metrics in a non-archimedean perspective:
\begin{lem}{(\cite[lem. 3.1, 3.2]{S})}
\\Let $z \in \X_0$ such that $\X_0$ is log terminal in a neighbourhood of $z$. There exists $\ell>0$ and a global section $\sigma \in H^0(\X_0, \ell K_{\X_0})$, non-vanishing at $z$ and such that the non-log terminal locus of $\X_0$ is contained in the vanishing locus of $\sigma$.
\\Furthermore, the singularities of $\psi_0$ are milder than any log poles, in the following sense: for any $\eps>0$ and $z$ as above, there exists $\sigma$ as above and $C>0$ such that:
$$C \ge \psi_0 \ge -C_\eps + \eps \log \lvert \sigma \rvert.$$ 
\end{lem}
\section{Convergence of measures}
\subsection{Proof of theorem A}
We first deduce theorem A from theorem B. This boils down to the following proposition: 
\begin{prop}
Let $\X_c \fl \de$ be a stable degeneration of canonically polarized manifolds, with central fiber $\X_{c, 0} = \sum_{i \in I} D_i$.
\\The Kähler-Einstein measures $\mu_t = \omega_{t}^n$ on $\X_c$ converge weakly to the measure $\omega_{KE}^n$ as $t \rightarrow 0$.
\\Furthermore, we have $\int_{D_i} \omega_{KE}^n  = (K_{\X_c}^n \cdot D_i)$ for any $i \in I$.
\end{prop}
\begin{proof}
As the measure $\omega^n_{KE}$ does not charge mass on singularities of $\X_{c,0}$, it is enough to show that for any open chart $\U$ of $\X_c \backslash \Sing(\X_{c,0})$ and $\eta \in \mathcal{C}^0_c(\U)$, we have: 
$$\int_{X_t} \eta \; \omega_{KE,t}^n \xrightarrow[t \rightarrow 0]{} \int_{\X_0^{sm}} \eta \; \omega_{KE}^n.$$
By the second item of proposition 2.1, we can find on $\U$ a family of biholomorphisms $\beta_t : U_0 \xrightarrow{\sim} U_t$ such that $\beta_t$ converge smoothly to identity, and such that $\beta_t^*\phi_t$ converge locally uniformly to $\phi_0$. 
\\The uniform bounds on $\sup \phi_t$ and $\sup \phi_0$ imply that the Kähler-Einstein measures $\omega_{t}^n = e^{\phi_t} \chi_t^n$ are uniformly bounded from above, and we conclude by dominated convergence.
\\The statement about the mass of $\omega^n_{KE}$ follows from remark 2.2.
\end{proof}
\subsection{Proof of theorem B}
Now let $X \fl \de^*$ be a meromorphic degeneration of compact complex manifolds, together with a normal model $\X_{\infty} \fl \de$ with reduced special fiber.
\\We assume that the family $(\mu_t)_{t \in \de^*}$ converges weakly on $\X_{\infty}$ to a measure $\mu_{\infty}$ on $\X_{\infty, 0}$, such that $\mu_{\infty}$ does not charge mass on the singular locus of $\X_0$.
\\We now want to compute the weak limit of the probability measures $\mu_t$ on $X^{\text{hyb}}$, under the assumption that we have weak convergence to $\mu_{\infty}$ on the normal model $\X_{\infty}$. Using the homeomorphism $X^{\text{hyb}} \simeq \varprojlim_{\X} \X^{\text{hyb}}$, this boils down to computing the weak limit of $\mu_t$ on $\X^{\text{hyb}}$ for every snc model $\X/ \de$ of $X$. 
\\Let $\X$ be an arbitrary model of $X$, and consider the closure $\Gamma$ inside $\X \times_{\de} \X_{\infty}$ of the graph of the rational map $\X \dashrightarrow \X_{\infty}$. Taking normalization of $\Gamma$ and  then applying Hironaka's resolution of singularities, we obtain an snc model $\X'$ of $X$, dominating both $\X$ and $\X_{\infty}$. This implies that snc models dominating $\X_{\infty}$ are cofinal in the category of models of $X$, so that to compute the hybrid limit of $\mu_t$'s, it is enough to compute its weak limit on $\X^{\text{hyb}}$ for any snc model $\X$ dominating $\X_{\infty}$. Let us fix such a model $\X$ for the rest of the section.
\\Write the central fiber of our model $\X_{\infty}$ of $X$ as $\X_{\infty, 0}= \sum_{ i \in I} D_i$. The total space $\X_{\infty}$ being normal, the local rings $\gO_{\X_{\infty}, \eta_{D_i}}$ are discrete valuation rings, so that each $D_i$ induces a divisorial valuation $v_{D_i} \in X^{\text{an}}$. These valuations are centered on the strict transforms $D'_i$ of the $D_i$ on $\X$.
\\The archimedean limit of the measures $\mu_t$ on $\X$ can be described as follows. Define the measure $h^*\mu_{\infty}$ on $\X_0$ by pulling back the top degree current $\mu_{\infty}$ on the smooth locus of $\X_{\infty, 0}$, and then extending the obtained measure by zero on $\X_0$. The fact that $\mu_{\infty}$ does not charge mass on the singularities of $\X_{0, \infty}$ immediately implies that the measures $\mu_t$ converge weakly on $\X$ to $h^*\mu_0$.
\\We will now compute the limit of the $\mu_t$'s on the hybrid space $\X^{\text{hyb}}$. Using \cite[lem. 3.6]{BJ} and a partition of unity argument, it is now enough to compute, for any $\U \subset \X$ open chart, the weak limit of the pushed-forward measures $(\Log_{\U})_* \mu_t$ when $t \rightarrow 0$. This is the purpose of the following lemma:
\begin{lem}
Let $\U \subset \X$ be an open subset such that $\U \cap \X_0 \subset \mathring{D}'_i$ for $i \in I$, and $\eta \in \mathcal{C}^0_c(\U)$. We have :
$$ (\Log_{\U})_*(\eta \mu_t) \xrightarrow[t \rightarrow 0]{} \bigg(\int_{\U \cap \X_0} \eta \; h^* \mu_{\infty} \bigg) \delta_{v_{D_i}}.$$
\end{lem}
\begin{proof}
If $\U \cap \X_0 \subset \mathring{D}'_i$, then the map $\Log_{\U}$ contracts $\U \backslash \X_0$ to $v_{D_i}$, and the measure $(\Log_{\U})_*(\eta \mu_t)$ is $\bigg( \int_{U_t} \eta \mu_t \bigg) \delta_{v_{D_i}}$, so that the above statement is a simple reformulation of the weak convergence of the $\mu_t$'s on $\X$.
\end{proof}
This allows us to deduce the following:
\begin{prop}
Let $\X \fl \de$ be an snc model of $X$, dominating $\X_{\infty}$.
\\Then the family of measures $(\mu_t)_{t \in \de^*}$ on $\X^{\emph{hyb}}$ converge weakly to the atomic measure: 
$$\mu_{0} = \sum_{i \in I} (\int_{D_i} \mu_{\infty}) \delta_{v_{D_i}}.$$
\end{prop}
This now implies that for any sequence $(t_j)_{j \in \N}$ going to $0$, which we may assume after extraction such that the sequence of measures $(\mu_{t_j})$ converges weakly to a limit $\mu'$, the inequality $\mu' \ge \sum_{i \in I} (\int_{D_i} \mu_{\infty}) \delta_{v_{D_i}}$ holds. These two measures having same total mass, we necessarily have equality $\mu' =\sum_{i \in I} (\int_{D_i}\mu_{\infty}) \delta_{v_{D_i}}$, for any sequence $(t_j)_{j \in \N}$, which concludes.
\begin{rem}
In the case where $\X_{\infty}$ is the canonical model of a degeneration of Kähler-Einstein manifolds, and is furthermore simple normal crossing (or more generally, toroidal), it is actually possible to prove directly that if $\U \subset \X$ is an open adapted chart in the sense of section 1, such that $\U$ meets a strata of $\X_0$ of codimension $>1$, then the measures $(\Log_{\U})_*(\eta \mu_t)$ converge weakly to $0$, using the estimates from \cite[prop. 3.1, lem. 3.2]{Zh}.
\end{rem}
\section{Applications and remarks}
\subsection{Continuity of solutions to the Monge-Ampère equations}
In this subsection, $X \fl \de^*$ is a degeneration of negatively curved Kähler-Einstein manifolds.
\\The Kähler-Einstein measures on $X_t$ tautologically satisfy the complex Monge-Ampère equation : $\omega_t^n = \mu_t$, whose right-hand side converges in the hybrid sense to the non-archimedean Monge-Ampère measure $\mu_0 = \MA(\phi_{K_{\X}})$. We deduce from this that if hybrid convergence of measures did imply strong convergence of solutions to the respective Monge-Ampère equation, then the sup-normalized Kähler-Einstein potentials $\psi'_t = \psi_t - \sup_{X_t} \psi_t$ would extend continuously to $X^{\text{hyb}}$, by setting $\Psi \equiv 0$ on $X_0^{\text{hyb}}$.
\\However proposition 2.1 implies that $\sup_{X_t} \psi_t$ is uniformly bounded from above, while $\inf_{X_t} \psi_t$ blows-up as $t \rightarrow 0$, which contradicts strong convergence inside of $X^{\text{hyb}}$.
\\We expect this phenomenon to be due to the fact that the archimedean behaviour of the metrics $\omega_t$ is too "nice" : viewing the family $(\omega_t)_{t \in \de^*}$ as a metric $\psi$ on the relative canonical bundle $K_{X / \de^*}$, it follows for instance from the classical work of Fujiki-Schumacher \cite{FS} that the curvature form $\Omega$ of $\psi$ - which satisfies $\Omega_{| X_t} = \omega_t$ - is semi-positive on $X$, that is, also along the horizontal direction. This in particular implies that the form $\Omega$ extends uniquely as a semi-positive $(1,1)$-current on any simple normal crossing model $\X$ of $X$. One can then attach to $\Omega$ a classical invariant called the Lelong number, in order to measure its singularities:
\begin{defn}
Let $\Omega$ be a semi-positive $(1,1)$-current on a smooth complex manifold $Z$, $D$ be a smooth prime divisor on $Z$, and $p$ a point of $D$.
\\If $z$ a local equation for $D$ at $p$, and $\phi$ a local potential near $p$ such that $\Omega = dd^c \phi$, then the Lelong number $\nu_D(\Omega, p)$ of $\Omega$ along $D$ is defined by the following formula:
$$ \nu_D( \Omega, p) = \max \{ a > 0 / \phi - a \log \lvert z \rvert \; \emph{is bounded from above} \}.$$
\end{defn}
The estimate of lemma 2.1 on the singularities of $\omega_{KE}$ near the nodes now implies that $\Omega$ has zero Lelong number along any irreducible component of the central fiber of a simple normal crossing model of $X$. However, the philosophy of \cite{BFJ0} suggests that all the non-archimedean data attached to the asymptotic behaviour of $\Omega$ at $t=0$ should be captured by the collection of those Lelong numbers along vertical divisors of simple normal crossing models of $X$, which would explain why the hybrid space $X^{\text{hyb}}$ is not able to see the difference between the Kähler-Einstein metrics $\omega_t$ and the reference Fubini-Study metrics $\theta_t$ - and similarly with the corresponding hermitian metrics on $K_X$.
\\Note however that we could expect a weaker form of convergence: is the 'hybrid' metric (in the sense of \cite{Fav}) induced on $X^{\text{hyb}}$ by the Kähler-Einstein metrics on $X$ and the canonical metric on $X^{\text{an}}$ semi-positive - even if not continuous ?
\\The question of strong convergence of solutions of Monge-Ampère equations on the hybrid space is mainly of interest in the Calabi-Yau case, where the behaviour of the metric is yet not understood in dimension $>2$. Given the above remarks, it is natural to assume that the degeneration is as "bad" as possible - that is, that the dimension of the essential skeleton $\Sk(X)$ is maximal. Those are the \emph{maximal degenerations} of Calabi-Yau manifolds, where the fibers $X_t$ approach the \emph{large complex structure limit} in the jargon of mirror symmetry. This also guarantees that the combinatorial information carried by $\Sk(X)$ is as rich as possible. 
\begin{qu}
Let $X \xrightarrow{\pi} \de^*$ a maximal degeneration of Calabi-Yau manifolds, polarized by a relatively ample line bundle $L$.
\\We denote by $\omega_t$ the unique Ricci-flat Kähler metric $\omega_t \in c_1(L_t)$, which is the curvature form of the hermitian metric $\phi_t$ on $L_t$, and by $\phi$ the unique metric solution to the non-archimedean Monge-Ampère equation provided by theorem 1.1:
$$\MA(\phi)= d\lambda,$$
$d\lambda$ being the Lebesgue measure on the essential skeleton $\emph{Sk}(X)$ of $X$, normalized so that $\int_{\emph{Sk}(X)} d\lambda= L^n$.
Is it true that the family of hermitian metrics $(\phi_t)_{t \in \de^*}$ on $L$ converges in the strong sense to $\phi$ ?
\end{qu}
\subsection{The Fano case}
Now let us discuss the case where $X \xrightarrow{\pi} \de^*$ is an algebraic degeneration of Fano manifolds.
\\In this case, the existence of a Kähler-Einstein metric on every fiber does \emph{not} hold in general, contrasting with the Calabi-Yau and canonically polarized cases. 
\\There are several obstructions to the existence of a Kähler-Einstein metric on a smooth Fano manifold $Z$, going back at least to Lichnerowicz and Matsushima, see for instance \cite{Fut}. Those "classical" obstructions all come from the automorphism group of the manifold $Z$, which must for instance be reductive in order for a Kähler-Einstein metric to possibly exist.
\\Further developments in the subject in the 90's have brought people to believe that the existence of a Kähler-Einstein metric on $Z$ is equivalent to a sophisticated form of GIT (Geometric Invariant Theory) stability for the polarized variety $(Z, -K_Z)$. The correct notion is the one of \emph{K-stability}.
\\Although we will not give a precise definition of K-stability here, it is worth emphasizing that it is a purely algebro-geometric condition on $(Z, -K_Z)$, which is furthermore Zariski open on the base of any holomorphic family of Fano manifolds. It is furthermore equivalent to the condition that $Z$ has finite automorphism group, in addition to being K-polystable, so that K-polystability should be thought as an extension of K-stability to manifolds with continuous families of automorphisms.
By \cite{CDS}, the K-polystability of $(Z, -K_Z)$ is now equivalent to the existence, of a Kähler-Einstein of positive curvature $\omega \in -c_1(K_{Z})$, satisfying the equation $\Ric(\omega)  = \omega$. This metric is furthermore unique up to $\Aut^0(Z)$, the connected component of the automorphism group of $Z$ containing the identity.
\\Thus, let $X \fl \de^*$ be from now on a degeneration of K-polystable Fano manifolds, hence Kähler-Einstein. In this case, positivity of the Ricci curvature and the classical Bonnet-Myers theorem ensure that the diameters of $(X_t, g_{\omega_t})$ satisfy a uniform upper bound, while the Bishop-Gromov inequality imply a uniform non-collapsing condition on the $(X_t, g_{\omega_t})$: for all small radius $r>0$ and $p \in X_t$, we have $\Vol_{g_t} B(p, r) \ge c r^{2n}$, $c>0$ being a uniform constant. 
\\This uniform non-collapsing condition is used in \cite{DS} to prove that the metric behaviour of the $X_t$ is rather nice : the limiting object is a mildly singular Fano variety, which also admits a singular Kähler-Einstein metric in the sense of \cite{EGZ}, and the Gromov-Hausdorff limit is the homeomorphic to the limit in the sense of algebraic geometry (i.e. in the relevant Hilbert scheme).
\\However, the unicity of a model $\X/ \de$ of $X$ whose central fiber matches the Gromov-Hausdorff limit of the $X_t$ no longer holds, as shown in the following example.
\begin{ex}
Consider $X = \CP^1 \times \de^*$, and $\X = \CP^1 \times \de$. We blow-up a closed point of $\X_0$ to obtain an snc model of $X$ whose central fiber is a chain of two $(-1)$-curves, and we contract the strict transform of $\X_0$ in this model, which yields a new snc model $\X'_0$ with central fiber $\CP^1$, which is not isomorphic to $\X$.
\end{ex}
In this case, even though the model $\X'$ has the 'right' central fiber, we see that the intrinsic behaviour of the metrics inside $\X'$ is rather badly behaved; the metric convergence is only extrinsic. This issue is due to the presence of a large automorphism group of the metrics, which makes the situation less rigid than in the negatively curved case. 
\\Even if we assume that the $X_t$'s have discrete autmorphism group for all $t \in \de^*$, it may very well happen that a model as above could be such that its irreducible central fiber $\X_0$ has non-trivial $\Aut^0(\X_0)$ and be non-unique, which would prevent us from determining on which one we could expect Cheeger-Gromov convergence.
\\However, it is proven in \cite{BX} that if there exists a model $\X/ \de$ whose irreducible central fiber is a K-stable $\Q$-Fano manifold, then such a model is actually unique - and so does the associated divisorial valuation $v = v_{\X_0}$. We then expect that the following holds: the Kähler-Einstein measures $(\mu_t)_{t \in \de^*}$ converge weakly on the hybrid space to the following:
$$\mu_0 = (-K_{X_t})^n \delta_{v}.$$
At the moment, we still lack uniform estimates on $\X$ on the Kähler-Einstein potentials with respect to a smooth reference family of metrics that would enable us to prove this, so that we leave this question aside for future work.
\bibliographystyle{alpha}
\bibliography{CVhybride.bbl}
\end{document}

%% file: macros2rue.tex
\newcommand{\ddbar}{\partial\bar{\partial}}
\newcommand \Z{\mathbb Z}
\newcommand \N{\mathbb N}
\newcommand \C{\mathbb C}
\newcommand \R{\mathbb R}
\newcommand \Q{\mathbb Q}
\newcommand \gO {\mathcal O}

\newcommand \M {\mathscr {M}}
\newcommand \Ld {\mathscr {L}}

\newcommand \de {\Delta}
\newcommand \U {\mathscr U}
\newcommand \CP {\mathbb P}
\newcommand \A {\mathbb A}

\newcommand \Sk {\text{Sk}}

\newcommand \X {\mathscr{X}}

\newcommand \cD {\mathscr{D}}

\newcommand \fl {\longrightarrow}
\newcommand \eps {\varepsilon}

\DeclareMathOperator{\Sp}{Spec}
\DeclareMathOperator{\Proj}{Proj}

\DeclareMathOperator{\ord}{ord}

\DeclareMathOperator{\Ker}{Ker}
\DeclareMathOperator{\Vol}{Vol}

\DeclareMathOperator \Id {Id}

\DeclareMathOperator{\PSH}{PSH}
\DeclareMathOperator{\Ric}{Ric}

\DeclareMathOperator{\Spec}{Spec}

\DeclareMathOperator{\Log}{Log}
\DeclareMathOperator{\Sing}{Sing}
\DeclareMathOperator{\MA}{MA}
\DeclareMathOperator{\val}{val}
\DeclareMathOperator{\Aut}{Aut}